\documentclass{amsart}
 \usepackage[latin1]{inputenc}
 \usepackage{graphicx}
 
\theoremstyle{plain} 
\newtheorem{teorema}{Theorem}[section] 
\newtheorem*{teoremas}{Theorem}
\newtheorem{prop}{Proposition}[section] 
\newtheorem{lemm}[prop]{Lemma}
\newtheorem{coro}{Corollary}[section] 

\theoremstyle{definition}
\newtheorem{rema}{Remark}[section] 

\def\R{\mathbb R}

\begin{document}

\title[On the integral formulas of Crofton and Hurwitz]{On the integral formulas of Crofton and Hurwitz relative to the visual angle of a convex set} 

\author{Juli\`a Cuf\'{\i}, Eduardo Gallego  \&  Agust\'{\i} Revent\'os}

\address{Departament de Matem\`atiques \\ 
Universitat Aut\`{o}noma de Barcelona\\ 
08193 Bella\-terra, Barcelona\\ Catalonia}
\email{jcufi@mat.uab.cat, egallego@mat.uab.cat, agusti@mat.uab.cat}

\thanks{The authors were  partially supported by grants 2017SGR358, 2017SGR1725 (Generalitat de Catalunya) and MTM2015-66165-P (Ministerio de Economía y Competitividad)}

\begin{abstract}
We provide a unified approach that encompasses some integral formulas for functions of the visual angle of a compact convex set  due to Crofton, Hurwitz and Masotti. The basic tool is an  integral formula  that also allows us to integrate  new functions  of the visual angle. 
As well we establish some upper and lower bounds for the considered integrals, generalizing in particular those obtained by Santal\'o for Masotti's integral.
\end{abstract}    

\maketitle

\section{Introduction}
In  1868 Crofton showed (\cite{crofton}), using arguments that nowadays belong to  integral geometry, the well known formula
$$
\int_{P\not \in K}(\omega-\sin\omega)\, dP={L^{2}\over 2}-\pi F,
$$
where $K$ is a planar compact convex set of area $F$, $L$ is the length of its boundary and 
$\omega=\omega(P)$ is the \emph{visual angle} of $K$ from the point $P$, that is the angle between the two tangents from $P$ to the boundary of $K$.

Later on, Hurwitz in  1902  in his celebrated paper \cite{Hurwitz1902} considered again the integral of some functions of the visual angle. In particular he gave a new proof of the Crofton formula using the Fourier series of the radius of curvature of the boundary of $K$. He also computed 
\begin{equation}\label{hurwitzsin3}
\int_{P\not\in K}\sin^{3}\omega\, dP=\frac34L^{2}+\frac14\pi^{2}{\gamma}_{2}^{2}
\end{equation} 
as well as  the integral of a family of special functions that enables him to show that the quantities $\gamma_{k}^{2}=\alpha_{k}^{2}+\beta_{k}^{2}$, where $\alpha_{k}$ and $\beta_{k}$ are the Fourier coefficients of the radius of curvature, are invariant with respect to rigid motions of $K$.

In 1955 Masotti (\cite{masotti2}) considered a Crofton-type formula computing  
$$
\int_{P\not \in K}(\omega^{2}-\sin^{2}\omega)\, dP
$$  
in terms of the area of $K$, the length of $\partial K$ and the Fourier coefficients of the radius of curvature of $\partial K$. Santal\'o in 1976 (\cite{santalo}) gave lower and upper bounds for the above integral.
\medskip	

In this paper we provide a unified approach that encompasses the previous results and  allows us to obtain new integral formulas for functions of the visual angle. The basic tool is an  integral formula  given by our first result.

\begin{teoremas}[\ref{fundamental}]
Let $K$ be a compact convex set with boundary of class $C^{2}$ and let~$L$ be the length of $\partial K$. 
Let $c_{k}^{2}=a_{k}^{2}+b_{k}^{2}$ where $a_{k}$, $b_{k}$ are the Fourier coefficients of the support function of $K$.  Then, for every continuous function of the visual angle $f(\omega)$ on $[0,\pi]$ such that
$f(\omega)=O(\omega^{3})$, as $\omega$ tends to zero, one has
\begin{equation*}%\label{ft311}
\begin{aligned}
&\int_{P\notin K}f(\omega)\,dP\\
&=\left(\int_{0}^{\pi}\frac{f(\omega)(1+\cos\omega)^{2}}{\sin^3\omega}\,d\omega\right)\frac{L^{2}}{2\pi}+\pi\sum_{k\geq 2} \left(\int_{0}^{\pi}\frac{f(\omega)h_{k}(\omega)}{\sin^3\omega}\,d\omega\right)c_{k}^{2},
\end{aligned}
\end{equation*}
where   $h_{k}$, for $k\geq 2$, are the   universal functions given  in \eqref{hdek}.
\end{teoremas}
Crofton's formula, the integral of the above mentioned special functions considered by Hurwitz and the Masotti integral formula follow directly from Theorem  \ref{fundamental}. Moreover we improve the lower bound  given by Santal\'o for Masotti's integral (see Corollary \ref{masotticor}).

As well the integral \eqref{hurwitzsin3} follows from Theorem \ref{fundamental}. Indeed, a more general result is obtained: we can compute the integral of any power of $\sin\omega$. The corresponding result is

\begin{teoremas}[\ref{teoremasinus}] 
Let $K$ be a compact convex set with boundary of class $C^{2}$ and length $L$. 
Write $c_{k}^{2}=a_{k}^{2}+b_{k}^{2}$ where $a_{k}$, $b_{k}$ are the Fourier coefficients of the support function of $K$. Then
\begin{multline*}%\label{formulasinus}
\int_{P\notin K}\sin^{m}\omega\,dP=\frac{\pi \,m!}{2^{m-1}(m-2)\Gamma(\frac{m+1}{2})^2} {L^{2}\over 2\pi}\\ 
+{\pi^{2}\ m!\over 2^{m-1}(m-2)}\sum_{k\geq 2,\, \textrm{even}}{(-1)^{\frac{k}{2}+1}(k^{2}-1)\over \Gamma({m+1+k\over 2})\Gamma({m+1-k\over 2})}c_{k}^{2}.
\end{multline*}
For $m$ odd the index $k$ in the sum runs only from $2$ to $m-1$.
\end{teoremas}

When $K$ is a compact convex set of constant width one has $c_{k}=0$ for $k$ even, so that the integral of $\sin^{m}\omega$ is, in this case,  $L^{2}$ multiplied by a factor that depends only on $m$. 

Next we extend the formulas of Crofton and Masotti by means of the equality
\begin{equation}\label{omegasinusm}
\int_{P\notin K}(\omega^{m}-\sin^{m}\omega)\,dP=-\pi^{m}F+M_{m}\frac{L^{2}}{2\pi}+\pi\sum_{k\geq 2}\beta_{k}c_{k}^{2},
\end{equation}
where the quantities $M_{m}$ and $\beta_{k}$   can be explicitly computed (see section \ref{sectionset}). For instance in the case $m=3$ we get
\begin{equation*}
\begin{split}
\int_{P\notin K}(\omega^{3}-\sin^{3}\omega)\,dP&=-\pi^{3}F+\biggl(12\pi\ln(2)-\frac{3\pi}{2}\biggr)\frac{L^{2}}{2\pi}\\
&\quad+
12\pi^{2}\biggl(\ln(2)-\frac{19}{16}\biggr)c_{2}^{2}-6\pi^{2}\sum_{k\geq 3}\left(\Psi\left(\frac{k+1}{2}\right)+\gamma\right)c_{k}^{2},
\end{split}
\end{equation*}
where $\Psi(x)$ is the digamma function $\Psi(x)=(\ln \Gamma(x))'$, and $\gamma$ is the Euler--Masche\-roni constant. 

Finally since the quantities $\beta_{k}$ appearing in \eqref{omegasinusm} are not easily handled we give upper and lower bounds for $\int_{P\notin K}(\omega^{m}-\sin^{m}\omega)\,dP$ that generalize those given by Santal\'o for $m=2$. More precisely we prove

\begin{teoremas}[\ref{upperbound}]
 Let $K$ be a compact convex set with boundary of class $\mathcal{C}^{2}$, area~$F$ and length of the boundary $L$, and let $\omega=\omega(P)$ be the visual angle from the point~$P$. Then
$$
\int_{P\notin K}(\omega^{m}-\sin^{m}\omega)\,dP\leq -\pi^{m}F+M_{m}\frac{L^{2}}{2\pi},\quad m\geq 1,
$$
where $M_{m}=\displaystyle\int_{0}^{\pi}\frac{(\omega^{m}-\sin^{m}\omega)'}{1-\cos\omega}\,d\omega$. Equality holds only for circles. 
\end{teoremas}

And for the case of constant width we get

\begin{teoremas}[\ref{ctwidth2}]
Let $K$ be a compact convex set of constant width, with boundary of class $\mathcal{C}^{2}$, of area~$F$ and length of the boundary~$L$, and let $\omega=\omega(P)$ be the visual angle from the point~$P$. Then
$$
\int_{P\notin K}(\omega^{m}-\sin^{m}\omega)\,dP\geq -\pi^{m}F+M_{m}\frac{L^{2}}{2\pi}-\frac{\pi^{m-1}}{4}\left(1-\biggl(\frac{3}{4}\biggr)^{m}\right)\Delta\geq 0,
$$
where $\Delta=L^{2}-4\pi F$ is the isoperimetric deficit. The first inequality becomes an equality only for circles.
\end{teoremas}

\section{Preliminaries}

A set $K\subset \R^{2}$ is \emph{convex} if it contains the complete segment joining every two points in the set. We shall consider nonempty  compact convex sets. The \emph{support function} of $K$ is defined as
$$
p_{K}(u)=\sup\{\langle x,u\rangle\, :\, x\in K\}\quad \text{for}\quad u\in \R^{2}.
$$
For a unit vector $u$ the number $p_{K}(u)$ is the signed distance of the support line to~$K$ with outer normal vector~$u$ from the origin. The distance is negative if and only if $u$points into the open half-plane containing the origin (cf.~\cite{Schneider2013}). 
We shall denote by~$p(\varphi)$ 
the $2\pi$-periodic function obtained by evaluating $p_{K}(u)$   on $u=(\cos\varphi,\sin\varphi)$. Note that $\partial K$ is the envelope of the one parametric family of lines given by
$$
x\cos\varphi+y \sin\varphi=p(\varphi).
$$
If the support function $p(\varphi)$ is differentiable  we can parametrize the boundary $\partial K$ by
\begin{equation*}\label{paramfi}
\gamma(\varphi)=p(\varphi)N(\varphi)+p'(\varphi)N'(\varphi),
\end{equation*}
where $N(\varphi)=(\cos\varphi,\sin\varphi).$
When $p$ is a $\mathcal{C}^{2}$ function the radius of curvature $\rho(\varphi)$ of $\partial K$ at the point $\gamma(\varphi)$ is given by $p(\varphi)+p''(\varphi)$. Then, convexity is equivalent to $p(\varphi)+p''(\varphi)\geq 0$. From now on we will assume that $p$ is of class $C^{2}$.	

It can be seen (cf.~\cite{santalo}) that the length $L$ of $\partial K$ and  the area $F$ of $K$ are given in terms of the support function, respectively, by
\begin{equation}\label{longarea}
L=\int_{0}^{2\pi}p\,d\varphi\quad\text{and}\quad F=\frac{1}{2}\int_{0}^{2\pi}(p^2-p'^2)\,d\varphi.
\end{equation}	
We will consider $\omega=\omega(P)$  the \emph{visual angle} of $\partial K$ from an exterior point $P$, that is the angle 
between the tangents from $P$ to $\partial K$.
For a function $f(\omega)$ of the visual angle $\omega$ we will deal with the integral of $f(\omega)$ with respect to the  area measure $dP$. 	

A special type of convex sets are those  of \emph{constant width}, that is those convex sets whose orthogonal projection on any direction have the same length~$w$.
In terms  of the support function~$p$ of $K$,  constant width means that $p(\varphi)+p(\varphi+\pi)=w$. Expanding $p$ in Fourier series
\begin{equation}\label{fourier}
p(\varphi)=a_{0}+\sum_{n=1}^{\infty}\left(a_{n}\cos n\varphi+b_{n}\sin n\varphi\right),
\end{equation}
it follows that
$$
p(\varphi)+p(\varphi+\pi)=2\sum_{n=0}^{\infty}\left(a_{2n}\cos 2n\varphi+b_{2n}\sin 2n\varphi\right),
$$ 
and therefore constant width is equivalent to   $a_{n}=b_{n}=0$ for all even $n>0$.

Given a compact convex set $K$ with support function $p(\varphi)$ 
the \emph{Steiner point} of~$K$ is defined by the vector-valued integral
$$
s(K)={1\over \pi}\int_{0}^{2\pi}p(\varphi)N(\varphi)\,d\varphi.
$$
This functional on the space of convex sets is additive with respect to the Minkowski sum. The Steiner point is rigid motion equivariant; this means that $s(gK)=gs(K)$ for every rigid motion~$g$. We remark that $s(K)$ can be considered, in the ${\mathcal C}^2$ case, as the centroid with respect to the curvature measure in the boundary $\partial K$; also we have that $s(K)$ lies in the interior of $K$ (cf.~\cite{groemer}). 
In terms of the Fourier coefficients of $p(\varphi)$ given in \eqref{fourier} 
the Steiner point is 
$$
s(K)=(a_{1},b_{1}).
$$

The relation between the support function $p(\varphi)$ of a convex set $K$ and the support function $q(\varphi)$ of the same convex set  but with respect to a new reference with origin at the point $(a,b)$, and axes parallel to the previous 
$x$ and $y$ axes,  is given by 
$$
q(\varphi)=p(\varphi)-a\cos\varphi-b\sin\varphi.
$$ 
Hence, taking the Steiner point as a new origin, we have
$$
q(\varphi)=a_{0}+\sum_{n\geq 2}\left(a_{n}\cos n\varphi+b_{n}\sin n\varphi\right).
$$
The associated \emph{pedal curve} to $K$ will be the curve that in polar coordinates with respect to the Steiner point as origin is given by $r=p(\varphi)$. 
In fact it is the geometrical locus of the orthogonal projection of the center on the tangents to the curve. The area~$A$ enclosed by the pedal curve is
$$
A= \frac12 \int_{0}^{2\pi}p(\varphi)^{2}\,d\varphi.
$$

\section{First integral formula}
Let  $f(\omega)$ be a function 
of the visual angle $\omega=\omega(P)$  of a given compact convex set~$K$ from a point~$P$ outside $K$.  In this section we will give a formula
to compute the integral of $f(\omega)$ with respect to the area measure~$dP$, $\int_{P\notin K}f(\omega)\,dP$,  in terms of the area of~$K$, the length of the boundary of~$K$, and the Fourier coefficients of the support function of~$K$. 
\medskip

For each point $P\notin K$ let $\varphi$ be the angle at the origin  formed  by  the normal to one of the tangents from $P$ to $\partial K$ with the $x$ axis; the pair $(\varphi,\omega)$
can be considered as a system of coordinates of $\R^2\setminus K$.

\begin{figure}[h] 
   \centering
   \includegraphics[width=.8\textwidth]{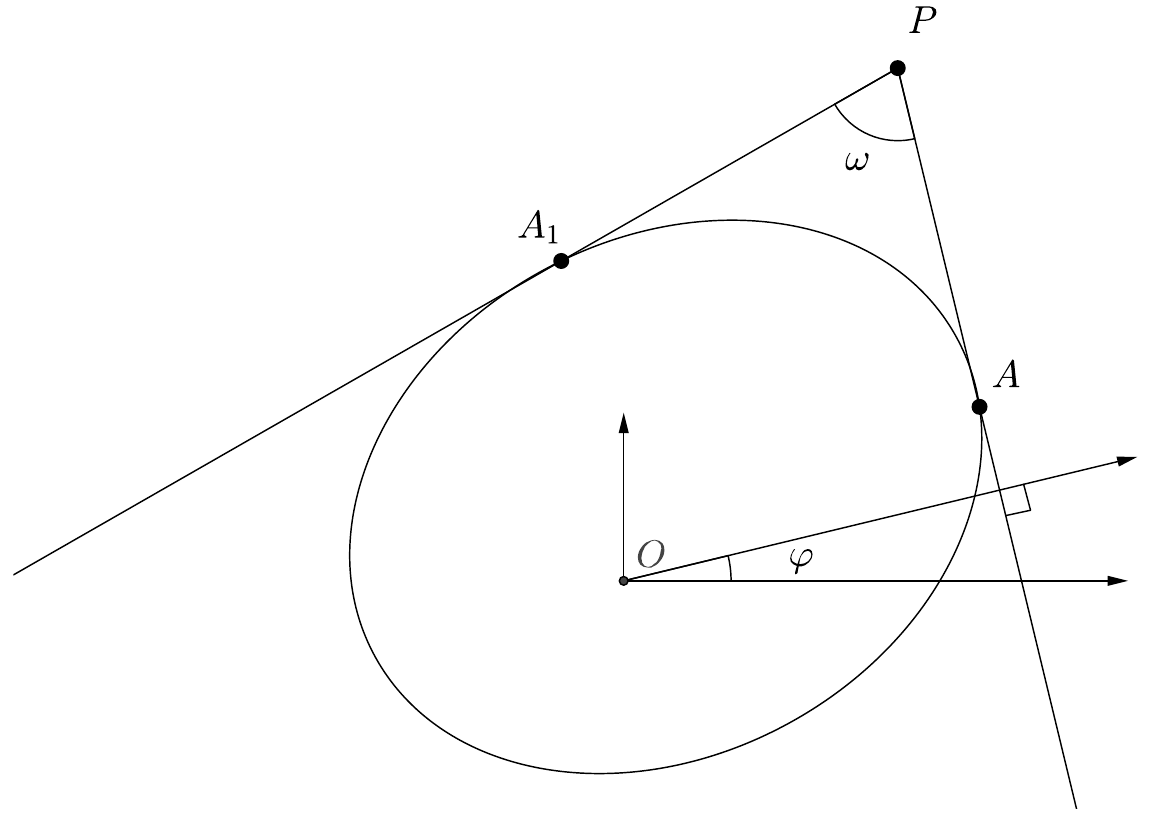} 
   \caption{The visual angle $\omega$.}
   \label{fig:masotti}
\end{figure}

Denoting by $A$, $A_{1}$ the contact points of the tangents from $P$
to $\partial K$, and by~$p=p(\varphi)$ the support function of $K$ with respect an origin $O$ inside $K$ (see Figure~\ref{fig:masotti}), we have 
\begin{equation*}
A=(p\cos\varphi-p'\sin\varphi, p\sin\varphi+p'\cos\varphi)
\end{equation*}
and hence, denoting 
$$
\varphi_{1}=\pi+\varphi-\omega
$$
we get
\begin{equation*}
A_{1}=(-p_{1}\cos(\varphi-\omega)+p'_{1}\sin(\varphi-\omega), -p_{1}\sin(\varphi-\omega)-p'_{1}\cos(\varphi-\omega)),
\end{equation*}
where $p_{1}(\varphi)=p(\varphi_{1})$, $p'_{1}(\varphi)=p'(\varphi_{1})$.

The intersection point $P=(X,Y)$ of the  tangent lines to $\partial K$ at points~$A$ and~$A_{1}$ is given by
\begin{align*}
X&=-\frac{1}{\sin\omega}(p\sin(\varphi-\omega)+p_{1}\sin\varphi),\\*[5pt]
Y&=\frac{1}{\sin\omega}(p\cos(\varphi-\omega)+p_{1}\cos\varphi).
\end{align*}
From this it is easy to see that the distances $T=PA$ and $T_{1}=PA_{1}$
are given by the positive quantities 
\begin{equation}\label{nov9}
\begin{aligned}
T &=\frac{1}{\sin\omega}(p\cos\omega-p'\sin\omega+p_{1}),\\*[5pt]
T_{1}&=\frac{1}{\sin\omega}(p_{1}\cos\omega+p_{1}'\sin\omega+p),
\end{aligned}
\end{equation}
due to the fact that the origin is inside $K$.

The area element $dP$ of $\R^2\setminus K$ is 
$$
dP=dX\wedge dY=\left(\frac{\partial X}{\partial \omega}\frac{\partial Y}{\partial\varphi}-\frac{\partial X}{\partial r\varphi}\frac{\partial Y}{\partial\omega}\right)\,d\varphi\wedge d\omega.
$$
A straightforward computation shows that 
\begin{equation*}\label{dP}
dP=\frac{TT_{1}}{\sin\omega}\,d\varphi\wedge d\omega.
\end{equation*}
This expression of the area element,  introduced by Crofton in \cite{crofton},    appears also in~\cite{masotti} and~\cite{santalo}.

Hence, the integral on $\R^2\setminus K$ of a suitable function of the visual angle $f(\omega)$ is given by
$$
\int_{P\notin K}f(\omega)\,dP=\int_{0}^{\pi}\int_{0}^{2\pi}\frac{f(\omega)}{\sin\omega}TT_{1}\,d\varphi\,d\omega=\int_{0}^{\pi}\frac{f(\omega)}{\sin\omega}\biggl(\int_{0}^{2\pi}TT_{1}\,d\varphi\biggr)\,d\omega.
$$

Now we will write the product $TT_{1}$ in terms of the Fourier coefficients of  $p(\varphi)$ given in \eqref{fourier},  
and the Fourier coefficients of $p_{1}(\varphi)$
given by
$$
p_{1}(\varphi)=a_{0}+\sum_{k>0} (A_{k}\cos k\varphi+B_{k}\sin k\varphi)
$$
which are related to the coefficients of $p(\varphi)$ by  
\begin{align*}
A_{k}&=(-1)^{k+1}(-a_{k}\cos k\omega+b_{k}\sin k\omega),\\
B_{k}&=(-1)^{k+1}(-a_{k}\sin k\omega-b_{k}\cos k\omega).
\end{align*}
Substituting these Fourier series in \eqref{nov9}, a straightforward but long calculation gives
\begin{equation*}\label{intP}
\int_{0}^{2\pi}TT_{1}\,d\varphi=\frac{1}{\sin^{2}\omega}\biggl(\frac{L^{2}}{2\pi}(1+\cos\omega)^{2}+\pi\sum_{k>0} c_{k}^{2}h_{k}(\omega)\biggr),
\end{equation*}
where $c_{k}^{2}=a_{k}^{2}+b_{k}^{2}$ and 
\begin{equation}\label{hdek}
\begin{split}
h_{k}(\omega)&=2\cos\omega+(-1)^{k+1}\bigl(-\cos k\omega(1+\cos^{2}\omega)\\
&\quad-2k\sin k\omega\sin\omega\cos\omega +k^{2}\cos k\omega\sin^{2}\omega\bigr).
\end{split}
\end{equation}
These functions can also be written as
\begin{equation*}
\begin{split}
h_{k}(\omega)&=\frac{(-1)^{k}}{4}\left[({k+1})^{2}\cos((k-2)\omega)\!+\!({k-1})^{2}\cos((k+2)\omega)\!-\!{2(k^{2}-3)}\cos(k\omega)\right]\\ 
&\quad+ 2\cos\omega.
\end{split}
\end{equation*}
Notice that $h_{1}\equiv 0$, $h_{k}(0)=2(1+(-1)^{k})$ and $h_{k}(\omega)=O((\omega-\pi)^{4})$, as $\omega$ tends to~$\pi$. Hence we have obtained the  following result.
\begin{teorema}\label{fundamental}
Let $K$ be a compact convex set with boundary of class $C^{2}$ and let $L$ be the length of $\partial K$. 
Let $c_{k}^{2}=a_{k}^{2}+b_{k}^{2}$ where $a_{k}$, $b_{k}$ are the Fourier coefficients of the support function of $K$. 
Then, for every continuous function of the visual angle~$f(\omega)$ on $[0,\pi]$ such that
$f(\omega)=O(\omega^{3})$, as $\omega$ tends to zero, one has
\begin{equation}\label{ft31}
\begin{aligned}
&\int_{P\notin K}f(\omega)\,dP\\
&=\left(\int_{0}^{\pi}\frac{f(\omega)(1+\cos\omega)^{2}}{\sin^3\omega}\,d\omega\right)\frac{L^{2}}{2\pi}+\pi\sum_{k\geq 2} \left(\int_{0}^{\pi}\frac{f(\omega)h_{k}(\omega)}{\sin^3\omega}\,d\omega\right)c_{k}^{2},
\end{aligned}
\end{equation}
where   $h_{k}$, for $k\geq 2$, are the   universal functions given  in \eqref{hdek}.
\end{teorema}
As a first example we can easily compute  the  integral in \eqref{hurwitzsin3}:
\begin{equation}\label{set23}
\begin{split}
\int_{P\notin K}\sin^{3}\omega\, dP&=\frac{L^{2}}{2\pi}\int_{0}^{{\pi}}(1+\cos\omega)^{2}d\omega+\pi\sum_{k\geq 2} c_{k}^{2}\int_{0}^{\pi}h_{k}(\omega)\,d\omega\\*[5pt]
&=\frac{3}{4}L^{2}+\frac{9}{4}\pi^{2}c_{2}^{2}.
\end{split}
\end{equation}
Notice that $\gamma_{2}^{2}=9c_{k}^{2}$ (see the footnote in page \pageref{hurwitz}).
This formula shows that the quantity  $c_{2}^{2}$ is invariant with respect to euclidean motions of $K$.

\section{The area of level sets and second integral formula}
As an application of Theorem \ref{fundamental} we will now compute the 
 area $F(\omega)$ enclosed by  the locus $C_{\omega}$ of the points from which  the convex set $K$ is viewed under the same angle $\omega$.  The corresponding  formula \eqref{areaF} was first given by Hurwitz in \cite{Hurwitz1902} and we will use it later on to obtain  another version 
of formula \eqref{ft31}.

Applying formula \eqref{ft31} with $f$ the characteristic function of the domain enclosed by the level set $C_{\omega}$ we have
\begin{equation}\label{nov9c}
\begin{aligned}
F(\omega) &= 
F+\left(\int_{\omega}^{\pi}\frac{(1+\cos\tau)^{2}}{\sin^{3}\tau}\,d\tau\right)\frac{L^{2}}{2\pi}+\pi\sum_{k\geq 2}\left(\int_{\omega}^{\pi}\frac{h_{k}(\tau)}{\sin^3{\tau}}\,d\tau\right)c_{k}^{2}\\*[5pt]
 &= F+\frac{L^2}{2\pi}\left[-\frac{1}{2\sin^2(\tau/2)}\right]_{\omega}^{\pi}+\pi\sum_{k\geq 2}\int_{\omega}^{\pi}\left(\frac{h_{k}(\tau)}{\sin^3\tau}\,d\tau\right) c_{k}^2\\*[5pt]
 &= F+\frac{L^2}{4\pi}\cot^{2}(\omega/2)+\pi\sum_{k\geq 2}\int_{\omega}^{\pi}\left(\frac{h_{k}(\tau)}{\sin^3\tau}\,d\tau\right) c_{k}^2.
\end{aligned}
\end{equation}
Using \eqref{longarea} and the Fourier expansions of $p$ and $p'$ one gets (see for instance \cite{CufiRev})
\begin{equation*}\label{areacks}
F=\frac{L^2}{4\pi}-\frac{\pi}{2}\sum_{k\geq 2}(k^2-1)c_{k}^2
\end{equation*}
and equation \eqref{nov9c} can be written as  
$$
F(\omega)=\frac{L^2}{4\pi}\frac{1}{\sin^2(\omega/2)}-\frac{\pi}{2}\sum_{k\geq 1}(k^2-1)c_{k}^2+\pi\sum_{k\geq 2}\int_{\omega}^{\pi}\left(\frac{h_{k}(\tau)}{\sin^3\tau}\,d\tau\right) c_{k}^2.
$$
Equivalently 
$$ 
F(\omega)\sin^2\omega=\frac{L^2}{2\pi}(1+\cos\omega)-\frac{\pi}{2}\sin^2\omega\sum_{k\geq 1}(k^2-1)c_{k}^2+\pi\sin^2{\omega}\sum_{k\geq 2}\int_{\omega}^{\pi}\left(\frac{h_{k}(\tau)}{\sin^3\tau}\,d\tau\right) c_{k}^2.
$$

Introducing the functions 
\begin{equation}\label{nov10e}
g_{k}(\omega)=1+\frac{(-1)^k}{2}\left((k+1)\cos(k-1)\omega-(k-1)\cos(k+1)\omega\right),
\end{equation}
already considered by Hurwitz, 
and using that 
$$
\frac{h_{k}(\tau)}{\sin^3\tau}=-\left(\frac{g_{k}(\tau)}{\sin^{2}\tau}\right)'
$$
we get

\begin{prop} 
Under the same hypothesis of Theorem \ref{fundamental}, the area $F(\omega)$  enclosed by  the locus $C_{\omega}$ of the points from which  the compact convex set $K$ is viewed under the same angle $\omega$ is given by
\begin{equation}\label{areaF}
F(\omega)\sin^2\omega=
\frac{L^2}{2\pi}(1+\cos\omega)+\pi\sum_{k\geq 2}c_{k}^2g_{k}(\omega),
\end{equation}
where the functions  $g_{k}(\omega)$ are defined in \eqref{nov10e}.
\end{prop}
Notice that the asymptotic behavior of $F(\omega)$ for $\omega$ near $0$ is given by 
\begin{equation}\label{nnn}
\lim_{\omega \to 0}\bigl(F(\omega)\sin^2\omega\bigr)=\frac{L^{2}}{\pi}+2\pi\sum_{k\geq 2,\, \textrm{even}}c_{k}^{2},
\end{equation}
an equality that appears in \cite{Hurwitz1902}. In the especial case of a  compact set  of constant width it is  
 $$
 \lim_{\omega \to 0}\bigl(F(\omega)\sin^2\omega\bigr)=\frac{L^{2}}{\pi}.
 $$

As $dP=P(\varphi,\omega)\,d\varphi\wedge d\omega$ we can write 
$$
F(\omega)=F+\int_{0}^{2\pi}\int_{\omega}^{\pi}P(\varphi,\tau)\,d\tau\, d\varphi,
$$ so that  $F'(\omega)=-\int_{0}^{2\pi}P(\varphi,\omega)\, d\varphi$. Therefore
$$
\int_{P\notin K}f(\omega)\,dP=-\int_{0}^{\pi}f(\omega)F'(\omega)\,d\omega.
$$
Integrating by parts and using the functions $g_{k}$ given in \eqref{nov10e} we obtain  

\begin{prop}\label{prop42}
With the same hypothesis of Theorem \ref{fundamental} we have 
\begin{equation*}\label{segonaexpbis}
\int_{P\notin K}f(\omega)\,dP=-\left[f(\omega)F(\omega)\right]_{0_{+}}^{\pi_{-}}
+ \frac{L^{2}}{2\pi}M(f) +
\pi\sum_{k\geq 2}\beta_{k}(f)c_{k}^{2},
\end{equation*}
where
\begin{equation}\label{defMbeta}
M(f)=\int_{0}^{\pi}{f'(\omega)\over 1-\cos\omega}\,d\omega\quad \text{and}\quad  \beta_{k}(f)=\int_{0}^{\pi}{f'(\omega)g_{k}(\omega)\over \sin^{2}\omega}\,d\omega.
\end{equation}
\end{prop}
Notice that $M(f)$ and $\beta_{k}(f)$ depend only on the function $f$ and not on the shape of the convex set $K$.
\medskip	

As an application of Proposition  \ref{prop42} we can easily prove  Crofton's formula
\begin{equation}\label{croftonform}
\int_{P\notin K}(\omega-\sin\omega)\,dP=-\pi F +{L^{2}\over 2}.
\end{equation}
Indeed $M(\omega-\sin\omega)=\pi$ and $\beta_{k}(\omega-\sin\omega)=\int_{0}^{\pi}g_{k}(x)/(1+\cos(x))\, dx=0$, as can be easily seen integrating by parts and using  elementary trigonometric identities. Since
$$
-\lim_{\omega\to \pi}f(\omega)F(\omega)+\lim_{\omega\to 0}f(\omega)F(\omega)=-\pi F
$$
the formula follows.

We will find now another expression for the universal factors  $g_{k}(\omega)/\sin^{2}(\omega)$ appearing in the integral defining the coefficients $\beta_{k}(f)$.   

\begin{lemm} The following identities hold
\begin{alignat*}{2}
\frac{g_{k}(\omega)}{\sin^{2}(\omega)}&=\frac{1}{1-\cos\omega}+2\sum_{j=1,\, \textrm{odd}}^{k-1}j\cos(j\omega),&\quad &\text{ for $k$ even},  \\*[5pt]
\frac{g_{k}(\omega)}{\sin^{2}(\omega)}&=-2\sum_{j=2,\, \textrm{even}}^{k-1}j\cos(j\omega), &\quad &\text{ for $k$ odd. } 
\end{alignat*}
\end{lemm}

\begin{proof}
From the expression of the conjugate Dirichlet kernel one has 
\begin{alignat}{2}\label{des22c}
\sum_{j=1, \,\textrm{odd}}^{k-1}\sin(j\omega)&={1-\cos(k\omega)\over 2 \sin \omega},&\quad &\text{ for $k$ even},
\intertext{and} 
\label{des22d}
\sum_{j=1,\, \textrm{even}}^{k-1}\sin(j\omega)&={\cos(\omega)-\cos(k\omega)\over 2 \sin \omega},&\quad &\text{ for $k$ odd}.\nonumber
\end{alignat}
Differentiating these  formulas the lemma follows.
\end{proof}
From this Lemma and Proposition \ref{prop42} we get

\begin{prop}\label{propo44} 
With the same hypothesis of Theorem \ref{fundamental} we have
\begin{equation}\label{set23x}
\begin{aligned}
\int_{P\notin K}f(\omega)\,dP  &=  -\left[f(\omega)F(\omega)\right]_{0}^{\pi}+ \frac{L^{2}}{2\pi}M(f) \\*[5pt]
 &\quad+  \pi\sum_{k\geq 2,\, \textrm{even}}\left(M(f) 
+ 2\sum_{j=1,\, \textrm{odd}}^{k-1}\int_{0}^{\pi}f'(\omega)j\cos(j\omega)\, d\omega \right)c_{k}^{2} \\*[5pt]
 &\quad+  \pi\sum_{k\geq 3,\, \textrm{odd}}\left( -2\sum_{j=2,\, \textrm{even}}^{k-1}\int_{0}^{\pi}f'(\omega)j\cos(j\omega) \,d\omega \right)c_{k}^{2},
\end{aligned}
\end{equation}
where $M(f)$ is given in \eqref{defMbeta}.
\end{prop}
This is a useful formula because it does not involve auxiliary functions and the coefficients of the $c_{k}^{2}$ do not depend on the convex set.

\section{Some applications of the integral formulas}

\subsection{Hurwitz functions}\label{HF}
In formula \eqref{set23x} it does not appear the Fourier coefficients of the support function of the compact set $K$, $a_{k}$, $b_{k}$ but only the quantities~$c_{k}^{2}=a_{k}^{2}+b_{k}^{2}$  for $k\geq 2$. So it will be interesting to see how  $c_{k}^{2}$ depends on the geometry of $K$. In fact Hurwitz in \cite{Hurwitz1902} found a formula relating the $c_{k}^{2}$ with the length of $\partial K$  and the integral outside $K$ of an elementary function of the visual angle. By a direct application of formula \eqref{set23x} we can prove
\begin{teorema}[Hurwitz, \cite{Hurwitz1902}]\label{hurwitz}
Let $K$ be a compact convex set  with boundary of class~$C^{2}$ and length~$L$. 
Let $c_{k}^{2}=a_{k}^{2}+b_{k}^{2}$ where $a_{k}, b_{k}$ are the Fourier coefficients of the support function of $K$. 
For the functions $f_{m}(\omega)$ given by
\begin{equation}\label{nov10d}
f_{m}(\omega)=-2\sin\omega+\frac{m+1}{m-1}\sin((m-1)\omega)-\frac{m-1}{m+1}\sin((m+1)\omega),
\end{equation}
we have\footnote{There is a misprint with the sign in Hurwitz's paper. 
Moreover the $c_{k}$ coefficients appearing in this formula are different from those in Hurwitz's paper because the latter correspond to the  Fourier series of the curvature radius function. In fact $\alpha_{k}=(1-k^{2})a_{k}$, $\beta_{k}=(1-k^{2})b_{k}$ where $\alpha_{k}$, $\beta_{k}$ are the Fourier coefficients of the radius of curvature.}
\begin{equation}\label{nov1}
\int_{P\not\in K} f_{m}(\omega)\,dP=L^2 + (-1)^m\, \pi^2 (m^2 - 1) c_m^2, \quad m\geq 2.
\end{equation}
\end{teorema}
\begin{proof}
In order to apply formula \eqref{set23x} we need to compute $[f_{m}(\omega)F(\omega)]_{0}^{\pi}$, $M(f_{m})$ and the integrals $\int_{0}^{\pi}f'_{m}(\omega)\cos(j\omega)\,d\omega$, for $j$ integer. 

First of all we have $[f_{m}(\omega)F(\omega)]_{0}^{\pi}=0$ since by  
\eqref{nnn}, 
$$
\lim_{\omega\to 0}f_{m}(\omega)F(\omega)=c \lim_{\omega\to 0}\frac{f_{m}(\omega)}{\sin^{2}\omega}=0.
$$ 
For  $M(f_{m})$ we need the  equalities
\begin{equation*}\label{nov10b}
f'_{m}(\omega)=2(1-\cos\omega)(1+2\sum_{j=1}^{m-1}j\cos(j\omega) + (m-1)\cos(m\omega)), \quad m\geq 2,
\end{equation*}  
which are obtained by direct computation using $2\cos(\omega)\cos(j\omega)=\cos(j+1)\omega+\cos(j-1)\omega.$ 
Then 
$$
M(f_{m})=\int_{0}^{\pi}{f'_{m}(\omega)\over 1-\cos\omega}\,d\omega=
2\int_{0}^{\pi}\biggl(1+2\sum_{j=1}^{m-1}j \cos(j\omega) + (m-1)\cos(m\omega)\biggr) \,d\omega =2\pi.
$$
Finally
\begin{equation*}
\begin{split}
\int_{0}^{\pi}f'_{m}(\sin(j\omega))'\,d\omega&=-\int_{0}^{\pi}f''_{m}\sin(j\omega)\,d\omega\\*[5pt]
&=-\int_{0}^{\pi}\biggl(2\sin\omega+2(m^2-1)\cos(m\omega)\sin\omega\biggr)\sin(j\omega)\,d\omega\\*[5pt]
&=-\pi\delta_{1,j}-2(m^2-1)\int_{0}^{\pi}\sin\omega\cos(m\omega)\sin(j\omega)\,d\omega\\*[5pt]
&=-\pi\delta_{1,j}-2(m^2-1)\frac{\pi}{4}(\delta_{j,m+1}-\delta_{j,m-1}),
\end{split}
\end{equation*}
where $\delta_{ij}=0$ for $i\neq j$ and $\delta_{ii}=1$.

Hence, if $k$ is even, 
\begin{equation*}
2\sum_{j=1,\, \textrm{odd}}^{k-1}\int_{0}^{\pi}f'_{k}(\sin(j\omega))'\,d\omega=-2\pi-(m^2-1)\pi(-\delta_{k-1,m-1})=-2\pi+(k^2-1)\pi\delta_{k,m}
\end{equation*}
since 
$$
\delta_{j+2,m+1}-\delta_{j, k-1}=0, \quad j=1,\dots,k-1.
$$
And, if $k$ is odd, 
\begin{equation*}
2\sum_{j= 2,\, \textrm{even}}^{k-1}\int_{0}^{\pi}f'_{m}(\sin(j\omega))'\,d\omega=-(m^2-1)\pi(-\delta_{k-1,m-1})=(m^2-1)\pi\delta_{k,m}.
\end{equation*}
Substituting in \eqref{set23x} 
we have
\begin{equation*}
\begin{split}
\int_{P\notin K}f_{m}(\omega)\,dP
&=L^2+\pi\sum_{k\geq 2,\, \textrm{even}}c_{k}^2
\biggl(2\pi-2\pi+(k^2-1)\pi\delta_{k,m}\biggr)\\*[5pt]
&\quad + \pi\sum_{k\geq 3,\, \textrm{odd}}c_{k}^2
\biggl(-(k^2-1)\pi\delta_{k,m}\biggr)\\*[5pt]
&= L^2+(-1)^m\pi^2c_{m}^2(m^2-1).
\end{split}
\end{equation*}
And the proof is finished.
\end{proof}
The theorem shows that the 	quantities $c_{m}^{2}$ are invariant with respect to euclidean motions of $K$.

Notice that the functions $f_{m}$ are related to the functions $g_{m}$ introduced in \eqref{nov10e}
by 
$$
g_{m}(\omega)=1+\frac{(-1)^{m}}{2}(f'_{m}(\omega)+2\cos(\omega)).
$$
\subsection{Masotti integral formula}
In \cite{masotti2} Masotti gives without proof a Crofton type formula evaluating  $\int_{P\notin K}(\omega^{2}-\sin^{2}\omega)\, dP$.  We will derive here Masotti's formula from \eqref{set23x}. 

Consider the function $f(\omega)=\omega^{2}-\sin^{2}\omega$ which clearly satisfies the hypothesis of  Theorem \ref{fundamental}, and let us compute $[f(\omega)F(\omega)]_{0}^{\pi}$, $M(f)$ and the integrals $\int_{0}^{\pi}f'(\omega)\cos(j\omega)\,d\omega$, for $j$ integer. 

We have
$$
\lim_{\omega\rightarrow \pi} f(\omega)F(\omega)=\pi^{2}F
$$
and
\begin{equation*}
\lim_{\omega\rightarrow 0} f(\omega)F(\omega)=\lim_{\omega\rightarrow 0}\frac{\omega^{2}-\sin^{2}\omega}{\sin^{2}\omega}\left(\frac{L^{2}}{2\pi}(1+\cos\omega)+\pi\sum_{k\geq 2}c_{k}^{2}g_{k}(\omega)\right)= 0,
\end{equation*}
since the term inside the parentheses is bounded. 
Hence, $\left[f(\omega)F(\omega)\right]_{0}^{\pi}=\pi^{2}F.$

On the other hand 
\begin{multline*}
M(f)=\int_{0}^{\pi}\frac{f'(\omega)}{1-\cos\omega}\,d\omega= \int_{0}^{\pi}\frac{2\omega-\sin(2\omega)}{1-\cos\omega}\,d\omega \\ 
= \left[\sin^{2}(\omega/2)-3\cos^{2}(\omega/2)-2\omega\cot(\omega/2)\right]_{0}^{\pi}=8.
\end{multline*}
Moreover, for $j\neq 2$, 
\begin{equation*}
\begin{split}
&\int_{0}^{\pi}f'(\omega)\cos(j\omega)\,d\omega=\int_{0}^{\pi}(2\omega-\sin(2\omega))\cos(j\omega)\,d\omega\\*[5pt]
=&\left[\frac{2}{j^{2}}(\cos(j\omega)+j\omega\sin(j\omega))-\frac{\cos((j-2)\omega)}{2(j-2)}+\frac{\cos((j+2)\omega)}{2(j+2)}\right]_{0}^{\pi}=\frac{8(1-(-1)^{j})}{j^{2}(j^{2}-4)},
\end{split}
\end{equation*}
and
$$
\int_{0}^{\pi}(2\omega-\sin(2\omega))\cos(2\omega)\,d\omega=0.
$$
It follows that
$$
\sum_{j=1,\, \textrm{odd}}^{k-1}\int_{0}^{\pi}f'(\omega)\cos(j\omega)\,d\omega=\sum_{j=1,\, \textrm{odd}}\frac{16}{j(j^{2}-4)}=\frac{4k^{2}}{1-k^{2}}.
$$
Summing up we obtain
\begin{teorema}[Masotti, \cite{masotti2}]\label{masottides}
Let $K$ be a compact convex set of area $F$ with boundary of class $C^{2}$ and length $L$. 
Let $c_{k}^{2}=a_{k}^{2}+b_{k}^{2}$ where $a_{k}$, $b_{k}$ are the Fourier coefficients of the support function of $K$. 
Then
\begin{equation}\label{20jb}
\int_{P\notin K}(\omega^{2}-\sin^{2}\omega)\,dP=-\pi^{2}F+\frac{4L^{2}}{\pi}+8\pi\sum_{k\geq 2,\,\textrm{even}}\left(\frac{1}{1-k^{2}}\right)c_{k}^{2}.
\end{equation}
Moreover the equality 
\begin{equation*}
\int_{P\notin K}(\omega^{2}-\sin^{2}\omega)\,dP=-\pi^{2}F+\frac{4L^{2}}{\pi}
\end{equation*}
holds if and only if the compact convex set $K$ has  constant width.
\end{teorema}
Besides the obvious inequality
\begin{equation*}
\int_{P\notin K}(\omega^{2}-\sin^{2}\omega)\,dP\leq -\pi^{2}F+\frac{4L^{2}}{\pi}
\end{equation*}
that follows from \eqref{20jb}, Santal\'o states in \cite{santalo} the lower bound
\begin{equation}\label{9juliol}
\int_{P\notin K}(\omega^{2}-\sin^{2}\omega)\,dP\geq  (16-\pi^{2})F,
\end{equation}
with equality only for circles.
We will improve now this last inequality.

\begin{teorema} 
Under the same hypothesis that in Theorem \ref{masottides} one has
\begin{equation}\label{26abr}
-\pi^{2}F+\frac{4L^{2}}{\pi}-\frac{4}{3}\Bigl(H-\frac{L^{2}}{\pi}\Bigr)\leq \int_{P\notin K}(\omega^{2}-\sin^{2}\omega)\,dP\leq -\pi^{2}F+\frac{4L^{2}}{\pi}, 
\end{equation}
where $H=\lim_{\omega \to 0}F(\omega)\sin^2\omega$ is given in \eqref{nnn}.
Equality in the left hand side holds if and only if $\partial K$ is a circle or a curve parallel to an astroid.
\end{teorema}

\begin{proof}
For the left-hand side just write 
$$
\int_{P\notin K}(\omega^{2}-\sin^{2}\omega)\,dP\geq -\pi^{2}F+\frac{4L^{2}}{\pi}-\frac{8\pi}{3}\sum_{k\geq 2,\, \textrm{even}} c_{k}^{2}=-\pi^{2}F+\frac{4L^{2}}{\pi}-\frac{4}{3}\Bigl(H-\frac{L^{2}}{\pi}\Bigr).
$$
Equality in the left-hand side holds 
if and only if the support function of $K$ with respect to the Steiner point is of the form $p(\varphi)=a_{0}+a_{2}\cos(2\varphi)+b_{2}\sin(2\varphi)$. This means that $\partial K$ is a circle or a curve parallel to an astroid (see for instance \cite{CufiRev}).
\end{proof}
In terms of the area $A$ of the pedal curve of $K$ with respect to the Steiner point we get 

\begin{coro}\label{masotticor}
Under the same hypothesis that in Theorem \ref{masottides}  one has
\begin{equation*}
\int_{P\notin K}(\omega^{2}-\sin^{2}\omega)\,dP\geq  (16-\pi^{2})F+\frac{32}{3}(A-F).
\end{equation*}
Equality  holds if and only if $\partial K$ is a circle or a curve parallel to an astroid.
\end{coro}

\begin{proof}
It is a simple consequence of \eqref{26abr} and the two inequalities $H-\dfrac{L^{2}}{\pi}\leq A-F $ and  $\Delta\geq 3\pi (A-F)$ (see \cite{CufiRev}). 
\end{proof}

The above argument shows that  \eqref{26abr} improves  Santal\'o's inequality \eqref{9juliol}.

\section{Integral of powers of the sinus of the visual angle}

In \eqref{set23} we have seen that 
$$
\int_{P\notin K}\sin^{3}\omega\, dP=\frac{3}{4}L^{2}+\frac{9}{4}\pi^{2}c_{2}^{2}.
$$
In this section we compute the integral of $\sin^{m}(\omega)$ for integer values of $m$  greater than $3$.
We have that 
$\left[\sin^{m}(\omega)F(\omega)\right]_{0}^{\pi}=0,$ so
$$
\int_{P\notin K}\sin^{m}(\omega)\,dP=M(\sin^{m}(\omega))\frac{L^{2}}{2\pi}+\pi\sum_{k\geq 2}\beta_{k}(\sin^{m}\omega)c_{k}^{2},
$$
with $M$ and $\beta_{k}$ are given in \eqref{defMbeta}.

As for $j$ even $\cos(\omega)\cos(j\omega)$ is an odd function with respect $\pi/2$ it can be seen from \eqref{set23x} that $\beta_{k}(\sin^{m}\omega)=0$ for every $k$ odd.

When $K$ is a convex set of constant width, $c_{k}=0$ for every even value of $k$, and we get
$$
\int_{P\notin K}\sin^{m}(\omega)\,dP=M(\sin^{m}(\omega))\frac{L^{2}}{2\pi}.
$$
Since  $M(\sin^{m}\omega)$ does not depend on the convex set $K$ we can compute this constant 
applying the above formula to the unit circle centered at the origin. 
If $r$ is the distance to the origin of the point $P$  we have
$\sin(\omega)=2\sin(\omega/2)\cos(\omega/2)=2{\sqrt{r^{2}-1}/r^{2}},$
and so
\begin{equation*}
\begin{split}
\int_{P\not \in K}\sin^{m}(\omega)\, dP &= \int_{0}^{2\pi}\int_{1}^{\infty}\sin^{m}(\omega)r\, dr\, d\theta=2\pi\, 2^{m}\int_{1}^{\infty}\left({\sqrt{r^{2}-1}\over r^{2}}\right)^{m}r\, dr\\
&= 2^{m}\pi B\left({m\over 2}+1, {m\over 2}-1\right),
\end{split}
\end{equation*}
where
$B(x,y)$ is the Beta function.
Hence 
$$
M(\sin^{m}\omega)=2^{m-1}B\left(\frac{m}{2}+1,\frac{m}{2}-1\right)=2^{m-1}\frac{\Gamma(\frac{m}{2}+1)\Gamma(\frac{m}{2}-1)}{(m-1)!}.
$$

Using the relation 
$$
\Gamma(z)\Gamma\biggl(z+\frac{1}{2}\biggr)=2^{1-2z}\sqrt{\pi}\Gamma(2z)
$$
we obtain
$$
\Gamma\biggl(\frac{m}{2}+1\biggr)\Gamma\biggl(\frac{m}{2}-1\biggr)=\frac{\pi m!(m-1)!}{2^{2m-2}(m-2)\Gamma(\frac{m+1}{2})^2},
$$
and hence
\begin{equation}\label{10maig}
M(\sin^{m}(\omega))=\frac{\pi \,m!}{2^{m-1}(m-2)\Gamma(\frac{m+1}{2})^2}.
\end{equation}
Notice that $M(\sin^{m}\omega)$ decreases with $m$, its maximum value $3\pi/2$ is attained for $m=3$  and it behaves as $1/\sqrt{m}$  when $m$ tends to infinity. 

We have proved the following 
\begin{prop}
Let $K$ be a compact convex of constant width  with boundary of class $C^{2}$ and length $L$ then
\begin{equation*}\label{ctwidth}
\int_{P\notin K}\sin^{m}\omega\,dP=\frac{\pi \,m!}{2^{m-1}(m-2)\Gamma(\frac{m+1}{2})^2}\;  {L^{2}\over 2\pi}.
\end{equation*}
\end{prop}
For general convex sets 
we have 
\begin{teorema} \label{teoremasinus}
Let $K$ be a compact convex set with boundary of class $C^{2}$ and length $L$. 
Write $c_{k}^{2}=a_{k}^{2}+b_{k}^{2}$ where $a_{k}$, $b_{k}$ are the Fourier coefficients of the support function of $K$. Then
\begin{equation}\label{formulasinus}
\int_{P\notin K}\sin^{m}\omega\,dP=M(\sin^{m}\omega) {L^{2}\over 2\pi}+
{m!\pi^{2}\over 2^{m-1}(m-2)}\sum_{k\geq 2, \,\textrm{even}}{(-1)^{\frac{k}{2}+1}(k^{2}-1)\over \Gamma({m+1+k\over 2})\Gamma({m+1-k\over 2})}c_{k}^{2},
\end{equation}
where $M(\sin^{m}(\omega))$ is given in \eqref{10maig}. For $m$ odd the index $k$ in the sum runs only from $2$ to $m-1$.
\end{teorema}

\begin{proof}
From \eqref{set23x} it is clear that we need to compute
$$
M(\sin^{m}\omega) 
+ 2\sum_{j=1, \,\textrm{odd}}^{k-1}\int_{0}^{\pi}(\sin^{m}\omega)'j\cos(j\omega) \,d\omega, \quad k\; \textrm{even}.
$$
Using \eqref{des22c} and integrating by parts one gets
\begin{multline*}
\sum_{j=1,\,\textrm{odd}}^{k-1}\int_{0}^{\pi}(\sin^m\omega)' (\sin(j\omega))'\,d\omega\\
=
-\int_{0}^{\pi}(m(m-1)\sin^{m-2}\omega-m^2\sin^m\omega)\frac{1-\cos(k\omega)}{2\sin\omega}\,d\omega.
\end{multline*}
Denoting $I_{m,k}=\int_{0}^{\pi}\sin^m\omega\cos(k\omega)\,d\omega$ 
we have
\begin{multline}\label{sumais}
\sum_{j=1,\,\textrm{odd}}^{k-1}\int_{0}^{\pi}(\sin^m\omega)' (\sin(j\omega))'\,d\omega \\
=
-\frac{m(m-1)}{2}I_{m-3,0}+\frac{m^2}{2}I_{m-1,0}+\frac{m(m-1)}{2}I_{m-3,k}-\frac{m^2}{2}I_{m-1,k}.
\end{multline}
By induction on $m$ and using  known relations of the Gamma function it can be seen  that
$$
I_{m,k}=(-1)^{k/2}\frac{2^{-m}m! \pi}{\Gamma(1+\frac{m-k}{2})\Gamma(1+\frac{m+k}{2})},
$$ (see for instance \cite[p.~372]{grads}).
Performing the operation on the right-hand side of~\eqref{sumais} with these values of $I_{m,k}$ we obtain
$$
\sum_{j=1,\,\textrm{odd}}^{k-1}\int_{0}^{\pi}\!(\sin^m\omega)' (\sin(j\omega))'\,d\omega
\!=\!-\frac{1}{2}M(\sin^{m}\omega)+{m!\pi\over 2^{m}(m-2)}{(-1)^{\frac{k}{2}+1}(k^{2}-1)\over \Gamma({m+1+k\over 2})\Gamma({m+1-k\over 2})}.\!
$$
From this, formula \eqref{formulasinus} follows. 

When $m$ is odd and $k>m$ we have that $m+1-k$ is an even non-positive integer and hence $\Gamma({m+1-k\over 2})=\infty$. From this remark the last assertion of  the theorem is proved.
\end{proof}
For instance,  for the special cases $m=3$, $4$ and $5$ we obtain the equalities 
\begin{align*}
\int_{P\notin K}\sin^{3}\omega\, dP&=\frac{3}{4}L^{2}+\frac{9}{4}\pi^{2}c_{2}^{2},\\*[5pt]
\int_{P\notin K}\sin^4\omega\, dP&=\frac{4}{3\pi}L^2+ \pi\sum_{k=2,\mathrm{even}}^{{\infty}} \frac{24}{9-k^2}\,c_{k}^{2},\\*[5pt]
\int_{P\notin K}\sin^5\omega\, dP&=\frac{5}{16}L^2+\frac{5\pi^2}{4}c_{2}^2-\frac{25\pi^2}{16}c_{4}^2.
\end{align*}
To end with we make the following remark. 
\begin{enumerate}
\item[a)] If $m=2r$ the coefficient of $c_{k}^{2}$ in \eqref{formulasinus} for $k\leq r$ is positive if and only if $k/2$ is odd. For $k>r$ this coefficient is positive if and only if $r$ is odd.
\item[b)] If $m=2r-1$ the coefficient of $c_{k}^{2}$ vanishes for $k> m$.
\end{enumerate}
\medskip

In \cite{Hurwitz1902} Hurwitz computed the integral of $\sin^{3}(\omega)$ and the  integrals of the functions~$f_{m}(\omega)$ given in \eqref{nov10d} without any relationship between them.  We will show now that the integrals of the powers of the sinus of the visual angle are a linear combination of  the integrals of the functions $f_{m}$. 

\begin{prop} For a compact convex set $K$ with boundary of class $C^{2}$ and~$m\geq 3$, we have
$$
\int_{P\notin K}\sin^{m}(\omega)\,dP=\frac{m!}{2^{m-1}(m-2)}\sum_{p=1}^{\infty}\frac{(-1)^{p+1}}{\Gamma(\frac{m+1}{2}+p)\Gamma(\frac{m+1}{2}-p)}\cdot \int_{P\notin K}f_{2p}(\omega)\,dP,
$$
where the functions $f_{2p}(\omega)$ are given in \eqref{nov10d}.
\end{prop}

\begin{proof}  
Substituting in \eqref{formulasinus} the value of $c_{k}^{2}$ given by \eqref{nov1} one gets 
\begin{equation}\label{coefldos}
\begin{aligned}
\int_{P\notin K}\sin^{m}\omega\,dP &=\frac{m!}{2^{m}(m-2)}\left(\frac{1}{\Gamma(\frac{m+1}{2})^{2}}-2\sum_{p=1}^{\infty}\frac{(-1)^{p+1}}{\Gamma(\frac{m+1}{2}+p)\Gamma(\frac{m+1}{2}-p)}\right)L^{2} \\*[5pt]
 &\quad+\frac{m!}{2^{m-1}(m-2)}\sum_{p=1}^{\infty}\frac{(-1)^{p+1}}{\Gamma(\frac{m+1}{2}+p)\Gamma(\frac{m+1}{2}-p)}\cdot \int_{P\notin K}f_{2p}(\omega)\,dP.
\end{aligned}
\end{equation}
Using the standard 
 notation for hypergeometric series we have
$$
2\sum_{p=1}^{\infty}\frac{(-1)^{p+1}}{\Gamma(\frac{m+1}{2}+p)\Gamma(\frac{m+1}{2}-p)}=2\frac{{}_{2}F_{1}\left(\frac{3-m}{2},1;\frac{3+m}{2};1\right)}{\Gamma(\frac{m+1}{2}+1)\Gamma(\frac{m+1}{2}-1)}
$$
and by the Gauss summation formula (see \cite[Vol III, p.~147]{GAUSS}) we obtain
$$
2\frac{{}_{2}F_{1}\left(\frac{3-m}{2},1;\frac{3+m}{2};1\right)}{\Gamma(\frac{m+1}{2}+1)\Gamma(\frac{m+1}{2}-1)}\!=\!\frac{2}{\Gamma(\frac{m+1}{2}+1)\Gamma(\frac{m+1}{2}-1)}\cdot \frac{\Gamma(m-1)\Gamma(\frac{m+3}{2})}{\Gamma(m)\Gamma(\frac{m+1}{2})}=\frac{1}{\Gamma(\frac{m+1}{2})^{2}}.
$$
Hence the coefficient of $L^{2}$ in \eqref{coefldos} vanishes.  
\end{proof}

\section{Extension of the Crofton and  Masotti  formulas and related inequalities}\label{sectionset}

In this section we consider the integral  
$$
\int_{P\notin K}(\omega^{m}-\sin^{m}\omega)\,dP,
$$ 
where $\omega$ is the visual angle of the convex set $K$ from the point $P$. For $m=1$ and $m=2$ these are the integrals appearing in Crofton's formula \eqref{croftonform} 
and in the Masotti integral formula \eqref{20jb},
 respectively.

For the general case,  we have by Proposition \ref{prop42}
\begin{equation}\label{ppp}
\int_{P\notin K}(\omega^{m}-\sin^{m}\omega)\,dP=-\pi^{m}F+M_{m}\frac{L^{2}}{2\pi}+\pi\sum_{k\geq 2}\beta_{k}c_{k}^{2},
\end{equation}
where $M_{m}=M(\omega^{m}-\sin^{m}\omega)$ and $\beta_{k}=\beta_{k}(\omega^{m}-\sin^{m}\omega)$ are given in \eqref{defMbeta}. The quantities  $M_{m}$  can be explicitly computed.   In fact $M(\sin^{m}\omega)$ is given in \eqref{10maig} and 
$$
M(\omega^{m})=2m (m-1)\pi^{m-2}\left({1\over m-2}+\sum_{k=1}^{\infty}{(-1)^{k}\pi^{2k}B_{2k}\over (m-2+2k)(2k)!}\right),
$$
where $B_{2k}$ are the Bernoulli numbers (see \cite[p.~189]{grads}). 
As the parenthesized expression tends to zero like $1/m^{2}$ when $m$ tends to infinity we see that $M(\omega^{m})$ behaves like $e^{m}$ and therefore  $M_{m}$ grows exponentially with $m$.
\medskip	

For $\beta_{k}$ recall that we have by Proposition \ref{propo44}
\begin{equation}\label{betask}
\beta_{k}=\begin{cases}\displaystyle M_{m}+2\sum_{j=1,\,\textrm{odd}}^{k-1}\int_{0}^{\pi}(\omega^{m}-\sin^{m}\omega)'j\cos(j\omega)\,d\omega, & \textrm{for}\; k\; \textrm{even},\\[.5cm]
\displaystyle -2\sum_{j=1,\,\textrm{even}}^{k-1}\int_{0}^{\pi}(\omega^{m})'j\cos(j\omega)\,d\omega, & \textrm{for}\; k\; \textrm{odd}.
\end{cases}
\end{equation}
Although the integrals appearing in the expression of the $\beta_{k}$
can be explicitly computed they are not easily handled. 

For instance in the case $m=3$ it can be seen that
\begin{equation*}
\begin{split}
\int_{P\notin K}(\omega^{3}-\sin^{3}\omega)\,dP&=-\pi^{3}F+\biggl(12\pi\ln(2)-\frac{3\pi}{2}\biggr)\frac{L^{2}}{2\pi}\\*[5pt]
&\quad+
12\pi^{2}\biggl(\ln(2)-\frac{19}{16}\biggr)c_{2}^{2}-6\pi^{2}\sum_{k\geq 3}\left(\Psi\left(\frac{k+1}{2}\right)+\gamma\right)c_{k}^{2},
\end{split}
\end{equation*}
where $\Psi(x)$ is the digamma function $\Psi(x)=(\ln \Gamma(x))'$, and $\gamma$ is the Euler--Masche\-roni constant. 

\subsection{Upper bounds} 
We obtain now an upper bound  for $\int_{P\notin K}(\omega^{m}-\sin^{m}\omega)\,dP$. 
For $m=3$, since $\Psi(x)>0$ for $x\geq 2$, we have 
\begin{equation}\label{30maig}
\int_{P\notin K}(\omega^{3}-\sin^{3}\omega)\,dP\leq -\pi^{3}F+\biggl(12\pi\ln(2)-\frac{3\pi}{2}\biggr)\frac{L^{2}}{2\pi}.
\end{equation}

For the general case  we obtain  the following result.

\begin{teorema}\label{upperbound} 
Let $K$ be a compact convex set with boundary of class~$\mathcal{C}^{2}$, area~$F$ and length of the boundary $L$, and let $\omega=\omega(P)$ be the visual angle from the point~$P$. Then
$$
\int_{P\notin K}(\omega^{m}-\sin^{m}\omega)\,dP\leq -\pi^{m}F+M_{m}\frac{L^{2}}{2\pi},\qquad m\geq 1,
$$
where $M_{m}=\displaystyle\int_{0}^{\pi}\frac{(\omega^{m}-\sin^{m}\omega)'}{1-\cos\omega}\,d\omega$. Equality holds only for circles. 
\end{teorema}

\begin{rema}
Since $M_{1}=\pi$, $M_{2}=8$ and $M_{3}=(12\ln(2)-3\pi/2)$ this result agrees with Crofton formula~\eqref{croftonform} for $m=1$, it  is a generalization of the upper bounds for Masotti's integral given in~\eqref{26abr} for $m=2$ and of the upper bound given in~\eqref{30maig} for~$m=3$.  
\end{rema}

\begin{proof}
We shall see that $\beta_{k}\leq 0$, for $k\geq 2$. Writing $V_{r,j}=\int_{0}^{\pi}\omega^{r}\cos(j\omega)\,d\omega$
the following recurrence formula can be checked 
$$
V_{r,j}=\frac{r}{j^{2}}\biggl((-1)^{j}\pi^{r-1}-(r-1)V_{r-2,j}\biggr).
$$
This gives  $V_{r,j}\geq 0$ for $j$ even and $V_{r,j}\leq 0$ for $j$ odd. 
As for $k$ odd we have $\beta_{k}=-2m\sum_{j=2,\,\textrm{even}}^{k-1}jV_{m-1,j}$ it follows that $\beta_{k}\leq 0$ for $k$ odd. 

For $k$ even, integrating by parts, we get
$$
\beta_{k}-\beta_{k+2}=2\int_{0}^{\pi}(\omega^{m}-\sin^{m}\omega)''\sin((k+1)\omega)\,d\omega.
$$
It can be seen that for $m\geq 4$ the function $\psi(\omega):=(\omega^{m}-\sin^{m}\omega)''$ is non-negative and increasing on the interval $[0,\pi].$ Then partitioning  $[0,\pi]$ by $t_{j}=\frac{j\pi}{k+1}$ with $j=0,\dots, k+1$, we get
\begin{multline*}
\int_{0}^{\pi}\psi (\omega )\sin(k+1)\omega 	\, d\omega \\*[5pt] 
=\sum_{j=0}^{k} \int_{t_j}^{t_{j+1}}\psi (\omega )\sin(k+1)\omega 	\, d\omega > \sum_{j=1}^{k} \int_{t_j}^{t_{j+1}}\psi (\omega )\sin(k+1)\omega \, d\omega \\*[5pt]
= \sum_{j=2,\, \textrm{even}}^{k} \left(\int_{t_j}^{t_{j+1}}\psi (\omega )\sin(k+1)\omega \, d\omega  - \int_{t_{j-1}}^{t_j}\psi (\omega )|\sin(k+1)\omega |\, d\omega \right)\\*[5pt]
= \sum_{j=2,\, \textrm{even}}^{k} \int_{t_j}^{t_{j+1}}\left(\psi (\omega )-\psi \bigl(\omega -{\pi\over k+1}\bigr)\right)\sin(k+1)\omega \, d\omega  > 0,
\end{multline*}
so that $\beta_{k}-\beta_{k+2}>0$ for all $k\geq 2$. Thus, in order to see that $\beta_{k}\leq 0$  it is enough to show that $\beta_{2}\leq 0$, with
$$
\beta_{2}=
\int_{0}^{\pi}(\omega^{m}-\sin^{m}\omega)'\left(\frac{1}{1-\cos\omega}+2\cos\omega\right)\,d\omega.
$$
For simplicity in the exposition we write $h_{m}(\omega )=\omega ^{m-1}-\sin^{m-1}\omega  \cos \omega $ and $\displaystyle g(\omega )={1\over 1-\cos \omega }+2\cos \omega .$ The function $g$ has exactly one root in the interval $[0,\pi]$, $\zeta=\arccos ((1-\sqrt{3})/2).$%\approx 1.94$. 

Notice that 
$$
h_{m}(\omega )g(\omega )=\omega ^{m-1}g(\omega )-\sin^{m-1}\omega \cos \omega \, g(\omega )\leq \omega ^{m-1}g(\omega ) +1
$$ 
for $0\leq \omega \leq \pi.$ Then, in order to see that $\beta_{2}<0$ it suffices to prove that $\int_{0}^{\pi}\omega ^{m-1}g(\omega )\,d\omega\leq -\pi$. First we see that the sequence $\int_{0}^{\pi}\omega ^{m-1}g(\omega )\,d\omega$ decreases with $m$, that is 
$$
\int_{0}^{\pi}(\omega ^{m}-\omega ^{m-1})g(\omega )\, d\omega =\int_{0}^{\pi}\omega ^{m-1}(\omega -1)g(\omega )\, d\omega 	\leq 0.
$$
The integrand is positive if $\omega \in [1,\zeta]$ and negative otherwise. Therefore if we see that 
$$
\int_{1}^{\zeta}\omega ^{m}(\omega -1)g(\omega )\, d\omega < \int_{\zeta}^{\pi}\omega ^{m-1}(\omega -1)|g(\omega )|\, d\omega 
$$ 
we are done. 

On the one hand $\int_{1}^{\zeta}\omega ^{m-1}(\omega -1)g(\omega )\, d\omega < \zeta^{m-1}(\zeta-1)\int_{1}^{\zeta}g(\omega )\, d\omega $. On the other hand 
$$
\int_{\zeta}^{\pi}\!\!\omega ^{m-1}(\omega -1)|g(\omega )|\!>\!\!\int_{\zeta+\varepsilon}^{\pi}\!\!\omega ^{m-1}(\omega -1)|g(\omega )|\, d\omega > (\zeta+\epsilon)^{m-1}(\zeta+\epsilon -1)\int_{\zeta+\epsilon}^{\pi}\!|g(\omega )|\, d\omega 
$$
for  $0<\epsilon<\pi-\zeta.$ Considering $\epsilon=1$ and using the fact that the function $2\sin \omega -\cot(\omega /2)$ is an antiderivative of $g$, a simple computation shows that
$$
\zeta^{m-1}(\zeta-1)\int_{1}^{\zeta}g(\omega )\, d\omega  < 
(\zeta+1)^{m-1}\zeta\int_{\zeta+1}^{\pi}|g(\omega )|\, d\omega 
$$
and the sequence $\int_{0}^{\pi}\omega ^{m-1}g(\omega )\,d\omega$ decreases with $m$. Since
$
\int_{0}^{\pi}\omega ^{2}g(\omega )\,d\omega <-\pi
$
we have proved that $\beta_{k}\leq \beta_{2}<0$. This gives us the upper bound
$$
\int_{P\notin K}(\omega^{m}-\sin^{m}\omega)\,dP\leq
-\pi^{m}F+M_{m}\frac{L^{2}}{2\pi}, \quad m\geq 4.
$$ 
As the cases $m=1,2,3$ are already known, this finishes the proof of the inequality in the theorem. Finally, since   the $\beta_{k}$ are not zero,  equality holds only when $c_{k}=0$ for $k\geq 2$, that is, when $\partial K$ is a circle.
\end{proof}

\subsection{Lower bounds}
For the case of constant width we have

\begin{teorema}\label{ctwidth2}
Let $K$ be a compact convex set of constant width, with boundary of class $\mathcal{C}^{2}$, of area $F$ and length of the boundary $L$, and let $\omega=\omega(P)$ be the visual angle from the point $P$. Then
$$
\int_{P\notin K}(\omega^{m}-\sin^{m}\omega)\,dP\geq -\pi^{m}F+M_{m}\frac{L^{2}}{2\pi}-\frac{\pi^{m-1}}{4}{\textstyle \left(1-\left(\frac{3}{4}\right)^{m}\right)}\Delta\geq 0,
$$
where $\Delta=L^{2}-4\pi F$ is the isoperimetric deficit. The first inequality becomes an equality only for circles.
\end{teorema}

\begin{proof}
In the constant width case it is $c_{k}=0$ for $k$ even and the only contribution in $\beta_{k}$ comes from $\omega^{m}$ when $k$ is odd.
Therefore since
$$
\int_{P\notin K}(\omega^{m}-\sin^{m}\omega)\,dP= -\pi^{m}F+M_{m}\frac{L^{2}}{2\pi}+\pi\sum_{k>2, \,\textrm{odd}}\beta_{k}c_{k}^{2}
$$
an upper bound for the positive quantity $K_{m}=-\pi\sum_{k>2,\, \textrm{odd}}\beta_{k}c_{k}^{2}$ will give a lower bound for $\int_{P\notin K}(\omega^{m}-\sin^{m}\omega)\,dP$.
Using \eqref{betask} we have  
$$
K_{m}=2\pi m\sum_{k>2,\, \textrm{odd}}\left(\sum_{j=2,\,\textrm{even}}^{k-1}j\int_{0}^{\pi}\omega ^{m-1}\cos(j\omega )\,d\omega \right)c_{k}^{2}.
$$
The following estimate holds
$$
\int_{0}^{\pi}\omega ^{m-1}\cos(j\omega )\,d\omega <\! \int_{\pi-{\pi\over 2j}}^{\pi}\omega ^{m-1}\cos(j\omega )\,d\omega \leq\! \int_{\pi-{\pi\over 2j}}^{\pi}\omega ^{m-1}\,d\omega ={\pi^{m}\over m}{\textstyle \left(1-\left(\frac{3}{4}\right)^{m}\right)}.
$$
Moreover  $\sum_{j=2,\,\textrm{even}}j=\frac14(k^{2}-1),$ so that
\begin{equation}\label{keme}
K_{m}\leq {\frac{\pi^{m+1}}{2}}{\textstyle \left(1-\left(\frac{3}{4}\right)^{m}\right)}\sum_{k>2,\,\textrm{odd}}(k^{2}-1)c_{k}^{2}.
\end{equation}
As $\Delta=2\pi^{2}\sum_{k\geq 2}(k^{2}-1)c_{k}^{2}$ (cf.\ \cite{CufiRev}) we have 
$$
K_{m}\leq {\pi^{m-1}\over 4}{\textstyle \left(1-\left(\frac{3}{4}\right)^{m}\right)}\Delta ,
$$
that gives the desired lower bound.

From formula \eqref{ppp} applied to a circle it follows that   $M_{m}\geq \pi^{m}/2$ 
and so  
we get 
\begin{equation*}
\begin{split}
\int_{P\notin K}(\omega^{m}-\sin^{m}\omega)\,dP&\geq -\pi^{m}F+M_{m}\frac{L^{2}}{2\pi}-{\pi^{m-1}\over 4}{\textstyle \left(1-\left(\frac{3}{4}\right)^{m}\right)}\Delta\\*[5pt]
&\geq 
{\pi^{m-1}\over 4}\left(\frac34\right)^{m}\Delta \geq 0.
\end{split}
\end{equation*}
Finally, if some $c_{k}\not =0$ we have strict inequality in \eqref{keme}, hence the first inequality becomes an equality only for circles.
\end{proof}

\bibliographystyle{smfplain}

\providecommand{\bysame}{\leavevmode ---\ }
\providecommand{\og}{``}
\providecommand{\fg}{''}
\providecommand{\smfandname}{\&}
\providecommand{\smfedsname}{\'eds.}
\providecommand{\smfedname}{\'ed.}
\providecommand{\smfmastersthesisname}{M\'emoire}
\providecommand{\smfphdthesisname}{Th\`ese}

\end{document}